\documentclass[11pt]{article}

\usepackage{amsmath,amssymb,amsthm,mathtools,color}
\usepackage[margin=1in]{geometry}
\usepackage{enumitem}

\newtheorem{theorem}{Theorem}[section]
\newtheorem{proposition}[theorem]{Proposition}
\newtheorem{corollary}[theorem]{Corollary}
\newtheorem{lemma}[theorem]{Lemma}

\newtheorem*{theorem*}{Theorem}

\newcommand{\Ric}{\operatorname{Ric}}
\newcommand{\Vol}{\operatorname{Vol}}

\newcommand{\dd}{\,\mathrm{d}}
\newcommand{\Sph}{\mathbb{S}}

\title{Weighted volume monotonicity and isoperimetric comparison under a lower Ricci curvature bound}
\author{Jiewon Park \and Keomkyo Seo}
\date{\today}

\begin{document}

\maketitle

\begin{abstract}
We establish a weighted extension of the Bishop-Gromov volume comparison theorem for complete Riemannian manifolds with Ricci curvature bounded below. Given a point $p$ and positive radial weights $f$ and $h$, we consider the weighted volume of a geodesic ball and a corresponding model weighted volume in the simply connected space form of constant sectional curvature $k$. We prove that, under the assumption that $h/f$ is nondecreasing, the ratio of these two weighted volumes is nonincreasing with respect to the radius. We also characterize the equality case, showing that equality forces the weight ratio to be constant and the corresponding geodesic ball to be isometric to the model ball. As applications, we derive weighted Bishop-Gromov-type comparisons, annular comparison inequalities, and weighted volume doubling estimates. We further apply the monotonicity formula to obtain weighted area-volume inequalities and sharp weighted isoperimetric-type comparisons for geodesic balls.
\\

\end{abstract}

\section{Introduction}

Let $(M^n,g)$ be an $n$-dimensional complete Riemannian manifold satisfying $\Ric_M\geq (n-1)kg$, and let $\mathbb M_k^n$ denote the simply connected $n$-dimensional space form of constant sectional curvature $k$. Then, for every $p\in M$, the celebrated Bishop-Gromov volume comparison theorem states that the function 
$$
\frac{\Vol_g\bigl(B_p(r)\bigr)}{\Vol_{\mathbb M_k^n}\bigl(B(r)\bigr)}
$$
is nonincreasing on $(0,R_k)$, where $R_k=\pi/\sqrt{k}$ when $k>0$ and $R_k=+\infty$ when $k\leq0$. In particular, $\Vol_g\bigl(B_p(r)\bigr)\leq\Vol_{\mathbb M_k^n}\bigl(B(r)\bigr)$ for every $r\in(0,R_k)$. Moreover, equality in the preceding volume inequality holds at some $r_0\in(0,R_k)$ if and only if $B_p(r_0)$ is isometric to the geodesic ball of radius $r_0$ in $\mathbb M_k^n$. This theorem and its rigidity consequences are basic tools in comparison geometry, volume convergence, compactness, and the structure theory of spaces with Ricci curvature bounded below (see  \cite{BishopCrittenden,CheegerColding,ColdingVolume,Gromov} for instance).

The Bishop-Gromov theorem may be viewed as a two-step argument, the first step being the Laplacian comparison and the second step an integration to give the desired monotonicity. In this paper, we show that the same strategy yields various weighted monotonicity formulas with general weights subject only to one relation, which then have applications such as weighted area and volume estimates, and in particular, a weighted isoperimetric-type inequality for geodesic balls. We allow two weights to be chosen independently, subject only to a single relation, whereas in existing results the weights are typically prescribed to the best of our knowledge, for example as powers of the Green function or in terms of a Bakry-\'Emery potential.

We assume that $\Ric_M\geq(n-1)kg$ where $k \in \mathbb{R}$. We fix $p\in M$, and write $\rho(x)=d(p,x)$ for the distance function and $\dd\mu(x)$ for the (unweighted) volume measure. For positive smooth functions $f,h\colon(0,R)\rightarrow(0,\infty)$, define
\begin{align*}
V_f(p,r) &:= \int_{B_p(r)}f(\rho(x))\dd\mu(x), \\
V_{k,h}(r) &:= \alpha_{n-1}\int_0^r h(t)s_k(t)^{n-1}\dd t,
\end{align*}
where $\alpha_{n-1}$ denotes the volume of the $(n-1)$-dimensional unit sphere in $\mathbb{R}^n$ and $s_k$ is the warping function of $\mathbb M_k^n$ (see Section 2 for the precise definition). Also let $A(p,r):=\mathcal H_g^{n-1}\bigl(S_p(r)\bigr)$. Our first result is the following weighted monotonicity formula.

\begin{theorem}[see Theorem \ref{thm:main}]
Let $(M^n,g)$ be an $n(\geq 2)$-dimensional complete Riemannian manifold satisfying $\Ric_M\geq(n-1)kg$ for \(k\in\mathbb{R}\). For \(p\in M\) and  \(0<R<R_k\), let $f,h\colon(0,R)\rightarrow(0,\infty)$ be positive smooth functions such that 
the functions $f(t)A(p,t), f(t)s_k(t)^{n-1},$ and $h(t)s_k(t)^{n-1}$ are integrable on \((0,R)\). Assume that $H(t):=\frac{h(t)}{f(t)}$ is nondecreasing on \((0,R)\). Then the function 
$$\Phi_{f,h}(p,r) := \frac{V_f(p,r)}{V_{k,h}(r)}$$
is nonincreasing in $r \in (0,R)$. Moreover, $\Phi_{f,h}(p,r_1)=\Phi_{f,h}(p,r_2)$ holds for some \(0<r_1<r_2<R\) if and only if \(H\) is constant on \((0,r_2)\) and $B_p(r_2)$ is isometric to the geodesic ball of radius $r_2$ in $\mathbb M_k^n$.
\end{theorem}

The classical Bishop-Gromov volume comparison theorem is recovered by taking $f(t)\equiv 1$ and $h(t)\equiv 1$. Thus the above result is a natural radially weighted extension. Moreover, it provides useful applications including weighted volume and area comparisons for balls and annuli, a weighted doubling estimate, and weighted isoperimetric-type inequalities. Among them, we give the following weighted isoperimetric-type result.

\begin{corollary}[see Corollary \ref{cor:ck-isoperimetric}]
Let $(M^n,g)$ be an $n(\geq 2)$-dimensional complete Riemannian manifold satisfying $\Ric_M\geq(n-1)kg$ for $k \in \mathbb{R}$. Fix \(p\in M\). Denote $v:=V_{c_k}(p,r)$ for $r \in (0, \frac{R_k}{2})$. Then we have
\[
A_{c_k}(p,r)
\leq
n\omega_n^{1/n}
v^{\frac{n-1}{n}}
\left[
1-
k\left(
\frac{v}{\omega_n}
\right)^{\frac{2}{n}}
\right]^{1/2},
\]
where $c_k=s_k'$ and $A_{c_k}(p,r)=c_k(r)A(p,r)$. Moreover, equality holds if and only if $B_p(r)$ is isometric to the geodesic ball of radius $r$ in $\mathbb{M}_k^n$.	
\end{corollary}

We now discuss related earlier work generalizing the classical Bishop-Gromov monotonicity to other settings together with their geometric and analytic implications. We first note that a weighted volume comparison theory for the Bakry-\'Emery Ricci tensor, generalizing the classical Bishop-Gromov comparison for the Ricci tensor, was developed by Qian \cite{Qian}, a diffusion-theoretic formulation of which was given by Bakry and Qian \cite{BakryQian}. Wei and Wylie \cite{WeiWylie} obtained comparison theorems under additional control of the potential. Comparisons for the weighted isoperimetric quotients and capacities, together with applications to parabolicity and hyperbolicity, were obtained by Hurtado, Palmer, and Rosales \cite{HurtadoPalmerRosales}.

Many new monotonicity formulas have been obtained over the past two decades using elliptic potentials rather than the distance function. Colding \cite{ColdingMonotonicity} established new monotonicity formulas involving the Green function that are closely related to the Bishop-Gromov volume comparison theorem. More precisely, let $M$ be a complete nonparabolic Riemannian manifold with nonnegative Ricci curvature and let $G$ be a positive Green function with a pole at a fixed point. By introducing an auxiliary function $b$ determined by $G$, Colding constructed a family of monotone quantities defined in terms of the level sets and sublevel sets of $b$ and involving powers of $|\nabla b|$. Further monotonicity formulas and their geometric applications were subsequently developed by Colding and Minicozzi in \cite{ColdingMinicozziSurvey,ColdingMinicozziHarmonic,ColdingMinicozziTangentCones}, and extended to Bakry-\'Emery Ricci curvature by Song, Wei, and Wu \cite{SongWeiWu}. Colding's monotonicity formulas also have applications to quantitative splitting \cite{BreinerPark}, where pinching of the monotone quantities $k+1$ independent points controls the distance to the nearest cone splitting into an $\mathbb{R}^k$-factor. In addition, we remark that a Hessian comparison for $b^2$ implies new monotonicity formulas as well \cite{Park}. See also \cite{Ng} for comparison results based on the modified Hessian for the squared distance function $\rho^2$ (not $b^2$).

The exterior capacitary monotonicities of Agostiniani, Fogagnolo, and Mazzieri \cite{AgostinianiFogagnoloMazzieri} are closely related to Colding's results, and yield a sharp Willmore inequality and a three-dimensional sharp isoperimetric inequality; a weighted version was developed by Wu and Wu \cite{WuWu}. In substatic geometry, Borghini and Fogagnolo \cite{BorghiniFogagnolo} proved monotonicity of normalized area and volume functionals built from a reparametrized distance, yielding a sharp weighted isoperimetric inequality.

Finally, we mention that monotonicity formulas and volume comparison have been extended beyond the classical smooth setting with pointwise curvature bounds. Chen and Wei \cite{ChenWei} proved improved Bishop-Gromov relative volume comparisons under integral Ricci curvature bounds. Gigli and Violo \cite{GigliViolo} generalized the family of Colding-Minicozzi monotonicity formulas for positive harmonic functions to $\operatorname{RCD}(0,N)$ spaces and established corresponding rigidity and almost rigidity results. Balogh and Krist\'aly \cite{BaloghKristaly} used optimal transport to prove the sharp isoperimetric inequality in $\operatorname{CD}(0,N)$ spaces. More recently, Carron, Mondello, and Tewodrose \cite{CarronMondelloTewodrose} proved a weak Bishop-Gromov monotonicity formula under a strong Kato bound on the negative part of the Ricci curvature. 

The remainder of the paper is organized as follows. In Section~2, we introduce the model functions associated with space forms of constant sectional curvature and recall the area form of the Bishop-Gromov comparison theorem. We then establish the main weighted monotonicity theorem and characterize its equality case. In Section~3, we develop several applications of the monotonicity formula, including results for particular choices of radial weights, weighted annular comparison inequalities, and volume doubling estimates. Finally, Section~4 is devoted to weighted area-volume inequalities and isoperimetric-type comparisons for geodesic balls.

\section{The weighted Bishop-Gromov-type theorem}

Let $(M^n,g)$ be a complete $n(\geq 2)$-dimensional Riemannian manifold with $\Ric_M\geq (n-1)kg$ for some $k\in\mathbb{R}$. Throughout the paper, we denote by $s_k$  and $c_k$ the functions defined by
\[
s_k(t)
=
\begin{cases}
\dfrac{1}{\sqrt{k}}\sin(\sqrt{k}\,t),
& k>0,\\[1.2ex]
t,
& k=0,\\[1.2ex]
\dfrac{1}{\sqrt{-k}}\sinh(\sqrt{-k}\,t),
& k<0
\end{cases}
\]
and
\[
c_k(t):=s_k'(t).
\]
Let
\[
R_k=
\begin{cases}
\dfrac{\pi}{\sqrt{k}},
& k>0,\\[1ex]
+\infty,
& k\leq0,
\end{cases}
\]
and
\[
\alpha_{n-1}
=
\left|\Sph^{n-1}\right|
=
n\omega_n,
\qquad
\omega_n
=
\left|\mathbb{B}^n\right|,
\]
where $\Sph^{n-1}\subset\mathbb{R}^n$ and $\mathbb{B}^n\subset\mathbb{R}^n$ denote the Euclidean unit sphere and unit ball, respectively. Fix a point $p\in M$ and let
\[
\rho(x):=d(p,x).
\]
We denote by $B_p(r):=\{x\in M:\rho(x)<r\}$ the geodesic ball centered at $p$ and by $S_p(r):=\partial B_p(r)$ the corresponding geodesic sphere. Define  
\begin{align*}
V(p,r)&:=\Vol_g\bigl(B_p(r)\bigr), \\
A(p,r)&:=\mathcal H_g^{n-1}\bigl(S_p(r)\bigr).
\end{align*}
Here $\mathcal H_g^{n-1}$ denotes the $(n-1)$-dimensional Hausdorff measure induced by $g$ on $S_p(r)$. The corresponding model area and volume functions are defined by
\begin{align*}
A_k(r) &:= \alpha_{n-1}s_k(r)^{n-1} \\
V_k(r) &:= \int_0^r A_k(t)\dd t = \alpha_{n-1}\int_0^r s_k(t)^{n-1}\dd t.
\end{align*}
The following well-known result is the Bishop-Gromov comparison theorem for geodesic spheres (see \cite{BishopCrittenden} for instance).

\begin{lemma} \label{lem:area-comparison}
Let \((M^n,g)\) be an $n(\geq 2)$-dimensional complete Riemannian manifold satisfying $\Ric_M\geq(n-1)kg$ for \(k\in\mathbb{R}\). Then, for $p\in M$, the function $\frac{A(p,r)}{A_k (r)}$ is nonincreasing on \((0,R_k)\), where $A(p,r)=\mathcal H_g^{n-1}(S_p(r))$. In particular, $A(p,r)\leq A_k(r)$ for every $r\in(0,R_k)$. Moreover, for any $r_0\in(0,R_k)$, equality holds if and only if $B_p(r_0)$ is isometric to the geodesic ball of radius $r_0$ in $\mathbb{M}_k^n$.
\end{lemma}

\noindent More generally, let $w\colon[0,R)\longrightarrow[0,\infty)$ for $0<R<R_k$ be a smooth function such that the functions $w(t)A(p,t)$ and $w(t)s_k(t)^{n-1}$ are integrable. Define the weighted volume and area functions by
\begin{align*}
V_w(p,r) &:= \int_{B_p(r)}w(\rho(x))\dd\mu(x) \\
A_w(p,r) &:= \int_{S_p(r)}w(\rho(x))\dd A(x) = w(r)A(p,r).
\end{align*}
Define the corresponding model weighted volume and area functions by
\begin{align*}
V_{k,w}(r) &:= \alpha_{n-1}\int_0^r w(t)s_k(t)^{n-1}\dd t \\
A_{k,w}(r) &:= \alpha_{n-1}w(r)s_k(r)^{n-1}.
\end{align*} 
By the coarea formula,
\[
V_w(p,r)
=
\int_0^r A_w(p,t)\dd t
=
\int_0^r w(t)A(p,t)\dd t,
\]
and thus
\[
\partial_r V_w(p,r)=A_w(p,r)=w(r)A(p,r)
\]
for almost every $r\in(0,R)$. Similarly,
\[
V_{k,w}'(r)=A_{k,w}(r).
\]

\noindent We are now ready to state our main theorem which can be regarded as a weighted Bishop-Gromov-type theorem.

\begin{theorem} \label{thm:main}
Let $(M^n,g)$ be an $n(\geq 2)$-dimensional complete Riemannian manifold satisfying $\Ric_M\geq(n-1)kg$ for \(k\in\mathbb{R}\). For \(p\in M\) and  \(0<R<R_k\), let $f,h\colon(0,R)\rightarrow(0,\infty)$ be positive smooth functions such that 
the functions $f(t)A(p,t), f(t)s_k(t)^{n-1}, h(t)s_k(t)^{n-1}$ are integrable on \((0,R_k)\). Assume that $H(t):=\frac{h(t)}{f(t)}$ is nondecreasing on \((0,R_k)\). Then the function 
$$\Phi_{f,h}(p,r) := \frac{V_f(p,r)}{V_{k,h}(r)}$$
 is nonincreasing in $r \in (0,R_k)$. Moreover, $\Phi_{f,h}(p,r_1)=\Phi_{f,h}(p,r_2)$ holds for some \(0<r_1<r_2<R_k\) if and only if \(H\) is constant on \((0,r_2)\) and $B_p(r_2)$ is isometric to the geodesic ball of radius $r_2$ in $\mathbb M_k^n$.
\end{theorem}

\begin{proof}
By the coarea formula,
\[
V_f(p,r)
=
\int_0^r f(t)A(p,t)\dd t.
\]
Define
\[
q_p(t):=\frac{A(p,t)}{s_k(t)^{n-1}}.
\]
Since $q_p$ is nonincreasing by Lemma~\ref{lem:area-comparison}, we see that
\begin{align}
V_f(p,r)
&=
\int_0^r f(t)s_k(t)^{n-1}q_p(t)\dd t
\nonumber\\
&\geq
q_p(r)
\int_0^r f(t)s_k(t)^{n-1}\dd t.
\label{eq:weighted-lower-bound}
\end{align}
The coarea formula yields at points of differentiability,
\begin{align*}
\partial_r V_f(p,r)&=A_f(p,r)=f(r)A(p,r), \\
V_{k,h}'(r)&=A_{k,h}(r) =\alpha_{n-1}h(r)s_k(r)^{n-1},
\end{align*}
which shows that
\begin{align*}
\partial_r\Phi_{f,h}(p,r)
&=
\frac{
 A_f(p,r)V_{k,h}(r)
 -A_{k,h}(r)V_f(p,r)
}{
 V_{k,h}(r)^2
}.
\end{align*}
Using \eqref{eq:weighted-lower-bound}, we obtain
\begin{align}
\partial_r\Phi_{f,h}(p,r)
&\leq
\frac{\alpha_{n-1}A(p,r)}{V_{k,h}(r)^2}
\left[
 f(r)\int_0^r h(t)s_k(t)^{n-1}\dd t
 -
 h(r)\int_0^r f(t)s_k(t)^{n-1}\dd t
\right]
\nonumber\\
&=
\frac{\alpha_{n-1}A(p,r)}{V_{k,h}(r)^2}
\int_0^r
 f(r)f(t)s_k(t)^{n-1}
 \bigl[H(t)-H(r)\bigr]\dd t
\nonumber\\
&\leq0,
\label{eq:original-proof-chain}
\end{align}
since $H$ is nondecreasing. This proves the desired monotonicity.

Now consider the case of equality. Suppose first that $\Phi_{f,h}(p,r_1)=\Phi_{f,h}(p,r_2)$ for some \(0<r_1<r_2<R_k\). Since \(\Phi_{f,h}(p,\cdot)\) is nonincreasing, it is constant on \([r_1,r_2]\). Hence we see that 
$$\partial_r\Phi_{f,h}(p,r)=0$$
for almost every \(r\in(r_1,r_2)\). It follows from \eqref{eq:original-proof-chain} that
$$H(t)=H(r)$$
for \(t\in(0,r)\). Since this holds for almost every $r\in(r_1,r_2)$ and $H$ is continuous, we conclude that $H$ is constant on $(0,r_2)$.

Since equality has to hold in (1) for every $r \in (r_1, r_2)$ and $t \in (0,r)$, we conclude that
$$
q_p(t)=q_p(r).
$$
Since $q_p$ is nonincreasing, it is constant on \((0,r)\). Together with
\[
\lim_{t\to0^+}q_p(t)=\alpha_{n-1},
\]
this gives
\begin{equation*}
A(p,t)
=
\alpha_{n-1}s_k(t)^{n-1}
=
A_k(t)
\qquad\text{for every }t\in(0,r_2).
\end{equation*}
By Lemma \ref{lem:area-comparison}, we see that $B_p(r_2)$ is isometric to the geodesic ball of radius $r_2$ in $\mathbb{M}_k^n$.

Conversely, suppose that \(H\equiv c\) on \((0,r_2)\) and that \(B_p(r_2)\) is isometric to the geodesic ball of radius $r_2$ in $\mathbb{M}_k^n$. Then, for every \(0<r\leq r_2\),
\[
V_f(p,r)
=
\alpha_{n-1}\int_0^r f(t)s_k(t)^{n-1}\dd t
\]
and
\[
V_{k,h}(r)
=
\alpha_{n-1}\int_0^r c f(t)s_k(t)^{n-1}\dd t
=
cV_f(p,r).
\]
Hence
\[
\Phi_{f,h}(p,r)=\frac1c
\qquad (0<r\leq r_2),
\]
and in particular
\[
\Phi_{f,h}(p,r_1)=\Phi_{f,h}(p,r_2),
\]
which completes the proof.

\end{proof}

\section{Applications of the main theorem}
In this section, we derive several consequences of Theorem \ref{thm:main} by taking various radial weights $f$ and $h$, including weighted volume comparisons, annular estimates, and doubling inequalities.

\begin{corollary}
\label{cor:ck}
Let $(M^n,g)$ be an $n(\geq 2)$-dimensional complete Riemannian manifold satisfying $\Ric_M\geq(n-1)kg$ for $k \geq0$. Define \(\rho(x)=d(p,x)\) for some point $p\in M$. Then the function
$$
\frac{V_{c_k}(p,r)}{V_k(r)}
$$
is nonincreasing on $\left(0,\frac{R_k}{2}\right)$.
\end{corollary}

\begin{proof}
Choose $f(t)=c_k(t)$ and $h(t)\equiv1$. Then $H(t)=\frac{1}{c_k(t)}$. Since $c_k'(t)=-ks_k(t)$, we have
\[
H'(t)
=
\frac{ks_k(t)}{c_k(t)^2}
\geq0
\]
when $k\geq0$. The conclusion follows from Theorem~\ref{thm:main}.
\end{proof}

Note that, when $k=0$, one has $c_0\equiv1$ and $V_0(r)=\omega_n r^n$. Hence Corollary~\ref{cor:ck} recovers the classical Bishop-Gromov theorem, which asserts that the normalized volume ratio $\frac{V(p,r)}{\omega_n r^n}$ is nonincreasing on $(0,+\infty)$. For $k<0$, however, $H(t)=\frac{1}{c_k(t)}$ is strictly decreasing, so Theorem~\ref{thm:main} alone is inconclusive.

\begin{corollary}
\label{cor:decreasing-f}
Let $(M^n,g)$ be an $n(\geq 2)$-dimensional complete Riemannian manifold satisfying $\Ric_M\geq(n-1)kg$ for $k \in \mathbb{R}$. Fix \(p\in M\). Let $f\colon(0,R_k)\rightarrow(0,\infty)$ be a positive smooth nonincreasing function such that the functions $f(t)A(p,t)$ and $f(t)s_k(t)^{n-1}$ are integrable on \((0,r)\) for every \(r\in(0,R_k)\). Then the function
$$
\frac{V_f(p,r)}{V_k(r)}
$$
is nonincreasing on \((0,R_k)\).
\end{corollary}

\begin{proof}
Take $h(t)\equiv1$. Then we see that $H(t)=\frac{1}{f(t)}$. Since $f$ is positive and nonincreasing, the conclusion follows directly from Theorem~\ref{thm:main}.
\end{proof}

The most natural choice is to take $h=f$. In this case, Theorem~\ref{thm:main} immediately yields the following weighted Bishop-Gromov comparison which holds for arbitrary $k$.

\begin{corollary} 
\label{cor:weighted-bg}
Let $(M^n,g)$ be an $n(\geq 2)$-dimensional complete Riemannian manifold satisfying $\Ric_M\geq(n-1)kg$ for $k \in \mathbb{R}$. Fix \(p\in M\). Let $f\colon(0,R_k)\rightarrow(0,\infty)$ be a positive smooth function such that, for every
\(r\in(0,R)\), the functions $f(t)A(p,t)$ and $f(t)s_k(t)^{n-1}$ are integrable on \((0,R_k)\). Then the function
$$
\Phi_f(p,r) := \frac{V_f(p,r)}{V_{k,f}(r)}
$$
is nonincreasing on \((0,R_k)\). If $\partial_r\Phi_f(p,r_0)=0$ for some $r_0 \in (0,R_k)$, then \(B_p(r_0)\) is isometric to the geodesic ball of radius $r_0$ in \(\mathbb M_k^n\).
\end{corollary}

\noindent It is worth noting that the choice $f\equiv1$ yields
$$
\Phi_1(p,r)=\frac{V(p,r)}{V_k(r)},
$$
and thus recovers the classical Bishop-Gromov volume comparison theorem. As another consequence of the weighted monotonicity formula, we obtain the following weighted annular comparison theorem.
\begin{corollary} 
\label{cor:weighted-annular}
Let $(M^n,g)$ be an $n(\geq 2)$-dimensional complete Riemannian manifold satisfying $\Ric_M\geq(n-1)kg$ for $k \in \mathbb{R}$. Fix \(p\in M\). Let $f\colon(0,R_0)\rightarrow(0,\infty)$ be a positive smooth function such that, for every
\(s\in(0,R_k)\), the functions $f(t)A(p,t)$ and $f(t)s_k(t)^{n-1}$ are integrable on \((0,R_k)\). Then, for every \(0<r<R<R_k\), we have
$$
\frac{V_f(p,R)-V_f(p,r)}{V_{k,f}(R)-V_{k,f}(r)} \leq \frac{V_f(p,r)} {V_{k,f}(r)}.
$$
\end{corollary}

\begin{proof}
Define
$$
q_p(t):=\frac{A(p,t)}{s_k(t)^{n-1}}
=
\alpha_{n-1}\frac{A(p,t)}{A_k(t)}.
$$
By Lemma~\ref{lem:area-comparison}, we see that $q_p$ is nonincreasing on $(0,R_k)$. For every $t\in(r,R)$, we have
\begin{align*}
V_f(p,R)-V_f(p,r)
&=
\int_r^R f(t)A(p,t)\dd t\\
&=
\int_r^R f(t)s_k(t)^{n-1}q_p(t)\dd t\\
&\leq
q_p(r)\int_r^R f(t)s_k(t)^{n-1}\dd t.
\end{align*}
On the other hand,
\begin{align*}
V_{k,f}(R)-V_{k,f}(r)
&=
\alpha_{n-1}
\int_r^R f(t)s_k(t)^{n-1}\dd t.
\end{align*}
Thus
\begin{align}
\frac{V_f(p,R)-V_f(p,r)}{V_{k,f}(R)-V_{k,f}(r)}\leq \frac{q_p(r)}{\alpha_{n-1}}. \label{ineq: 4}
\end{align}
Since
\begin{align*}
V_f(p,r)
&=
\int_0^r f(t)A(p,t)\dd t\\
&=
\int_0^r f(t)s_k(t)^{n-1}q_p(t)\dd t\\
&\geq
q_p(r)\int_0^r f(t)s_k(t)^{n-1}\dd t\\
&=
\frac{q_p(r)}{\alpha_{n-1}}V_{k,f}(r),
\end{align*}
we have
\begin{align}
\frac{V_f(p,r)}{V_{k,f}(r)} \geq \frac{q_p(r)}{\alpha_{n-1}}. \label{ineq: 5}
\end{align}
Combining (\ref{ineq: 4}) and (\ref{ineq: 5}) gives
$$
\frac{V_f(p,R)-V_f(p,r)}
{V_{k,f}(R)-V_{k,f}(r)}
\leq
\frac{q_p(r)}{\alpha_{n-1}}
\leq
\frac{V_f(p,r)}
{V_{k,f}(r)},
$$
which completes the proof.
\end{proof}

As a further consequence, we obtain the following weighted volume doubling estimate.

\begin{corollary} \label{cor:weighted-doubling}
Let $(M^n,g)$ be an $n(\geq 2)$-dimensional complete Riemannian manifold satisfying $\Ric_M\geq0$. Define \(\rho(x)=d(p,x)\) for some point $p\in M$. For \(\beta\geq0\), define
\[
V_\beta(p,r) := \int_{B_p(r)}\rho(x)^\beta\dd\mu(x).
\]
Then, for every \(\lambda\geq1\) and every \(r>0\),
\[
V_\beta(p,\lambda r)
\leq
\lambda^{n+\beta}V_\beta(p,r).
\]
In particular,
\[
V_\beta(p,2r)
\leq
2^{n+\beta}V_\beta(p,r).
\]
\end{corollary}

\begin{proof}
Take $k=0$ and $f(t):=t^\beta$.  Since $s_0(t)=t$, we have
\begin{align*}
V_{0,f}(r)
&=
\alpha_{n-1}
\int_0^r f(t)s_0(t)^{n-1}\dd t\\
&=
\alpha_{n-1}
\int_0^r t^{n+\beta-1}\dd t\\
&=
\frac{\alpha_{n-1}}{n+\beta}r^{n+\beta}.
\end{align*}
Moreover, by the coarea formula,
\begin{align*}
V_f(p,r)
&=
\int_0^r f(t)A(p,t)\dd t\\
&=
\int_0^r t^\beta A(p,t)\dd t\\
&=
\int_{B_p(r)}\rho(x)^\beta\dd\mu(x)\\
&=
V_\beta(p,r).
\end{align*}
By Corollary~\ref{cor:weighted-bg}, the function
$$
\frac{V_f(p,r)}{V_{0,f}(r)}
$$
is nonincreasing. It follows that
$$
\frac{V_f(p,\lambda r)}{V_{0,f}(\lambda r)}
\leq
\frac{V_f(p,r)}{V_{0,f}(r)}.
$$
Therefore,
\begin{align*}
\frac{V_\beta(p,\lambda r)}
{V_\beta(p,r)}
&=
\frac{V_f(p,\lambda r)}
{V_f(p,r)}\\
&\leq
\frac{V_{0,f}(\lambda r)}
{V_{0,f}(r)}\\
&=
\frac{(\lambda r)^{n+\beta}}
{r^{n+\beta}}\\
&=
\lambda^{n+\beta},
\end{align*}
which implies
$$
V_\beta(p,\lambda r)
\leq
\lambda^{n+\beta}V_\beta(p,r),
$$
which completes the proof.
\end{proof}

\section{Weighted isoperimetric inequalities for geodesic balls} \label{sec:weighted-isoperimetric}

In this section, we derive weighted area-volume comparisons and weighted isoperimetric inequalities for geodesic balls from Theorem \ref{thm:main}.

\begin{proposition}
\label{prop:weighted-area-volume}
Let $(M^n,g)$ be an $n(\geq 2)$-dimensional complete Riemannian manifold satisfying $\Ric_M\geq(n-1)kg$ for $k \in \mathbb{R}$. Fix \(p\in M\). Let $f\colon(0,R_k)\rightarrow(0,\infty)$ be a positive smooth function such that, for every
\(r\in(0,R_k)\), the functions $f(t)A(p,t)$ and $f(t)s_k(t)^{n-1}$ are integrable on \((0,R_k)\). Then 
$$ \frac{A_f(p,r)}{V_f(p,r)} \leq \frac{A_{k,f}(r)}{V_{k,f}(r)}.$$
\end{proposition}

\begin{proof}
Corollary~\ref{cor:weighted-bg} shows that the function
$$
\frac{V_f(p,r)}
{V_{k,f}(r)}
$$
is nonincreasing. Differentiating gives
$$
0
\geq
\frac{
A_f(p,r)V_{k,f}(r)
-
V_f(p,r)A_{k,f}(r)
}{
V_{k,f}(r)^2
},
$$
for almost every $r$. We then conclude by continuity.
\end{proof}

Assuming that the ratio of the model weighted area to the model weighted volume is nonincreasing, the weighted area of a geodesic sphere is bounded above by that of the model sphere enclosing the same weighted volume. This yields the following weighted isoperimetric-type comparison for geodesic balls.
\begin{theorem} \label{thm:weighted-isoperimetric}
Let $(M^n,g)$ be an $n(\geq 2)$-dimensional complete Riemannian manifold satisfying $\Ric_M\geq(n-1)kg$ for $k \in \mathbb{R}$. Fix \(p\in M\). Let $f\colon(0,R_k)\rightarrow(0,\infty)$ be a positive smooth function such that, for every
\(r\in(0,R)\), the functions $f(t)A(p,t)$ and $f(t)s_k(t)^{n-1}$ are integrable on \((0,R_k)\). Assume that
\[
\mathcal{Q}_{k,f}(r)
:=
\frac{A_{k,f}(r)}
{V_{k,f}(r)}
\]
is nonincreasing on \((0,R_k)\). For every $0<v<V_{k,f}(R_k)$, define $r_{k,f}(v):=V_{k,f}^{-1}(v)$ and $\mathcal{I}_{k,f}(v):= A_{k,f}\bigl(r_{k,f}(v)\bigr)$. Then we have
\[
A_f(p,r)
\leq
\mathcal{I}_{k,f}
\bigl(V_f(p,r)\bigr).
\]
Moreover, equality holds  for a fixed $r\in(0,R_k)$ if and only if $B_p(r)$ is isometric to the geodesic ball of radius $r$ in $\mathbb{M}_k^n$.
\end{theorem}

\begin{proof}
By Lemma~\ref{lem:area-comparison}, the function $q_p(t):=\frac{A(p,t)}{s_k(t)^{n-1}}$ is nonincreasing on $(0,R_k)$. In particular, since $\lim_{t\to 0^+}q_p(t)=\alpha_{n-1}$, we see that $q_p(t)\leq \alpha_{n-1}$ for every $t\in(0,R_k)$. By the coarea formula,
\begin{align}
V_f(p,r)
&=\int_0^r f(t)A(p,t)\,dt \nonumber \\
&=\int_0^r f(t)s_k(t)^{n-1}q_p(t)\,dt \nonumber\\
&\leq
\alpha_{n-1}\int_0^r f(t)s_k(t)^{n-1}\,dt \nonumber\\
&=V_{k,f}(r). \label{ineq: 2}
\end{align}
Furthermore, since $q_p$ is nonincreasing, one has $q_p(t)\geq q_p(r)$ for every $t\in(0,r)$. Therefore,
\begin{align}
V_f(p,r)
&=\int_0^r f(t)s_k(t)^{n-1}q_p(t)\,dt \nonumber \\
&\geq
q_p(r)\int_0^r f(t)s_k(t)^{n-1}\,dt. \label{ineq: 1}
\end{align}
Since
$$
A_f(p,r)
=
f(r)A(p,r)
=
f(r)s_k(r)^{n-1}q_p(r),
$$
(\ref{ineq: 1}) gives
\begin{align}
A_f(p,r) &\leq f(r)s_k(r)^{n-1}\frac{V_f(p,r)}{\displaystyle\int_0^r f(t)s_k(t)^{n-1}\,dt} \nonumber\\
&=\frac{\alpha_{n-1}f(r)s_k(r)^{n-1}}{\displaystyle\alpha_{n-1}\int_0^r f(t)s_k(t)^{n-1}\,dt} V_f(p,r) \nonumber \\
&=\frac{A_{k,f}(r)}{V_{k,f}(r)}V_f(p,r) \nonumber \\
&=\mathcal{Q}_{k,f}(r)V_f(p,r). \label{ineq: 3}
\end{align}
Let $v:=V_f(p,r)$ and $\sigma:=r_{k,f}(v)=V_{k,f}^{-1}(v)$. (\ref{ineq: 2}) implies that
$$
V_{k,f}(\sigma) = V_f(p,r)\leq V_{k,f}(r).
$$
Since the function $V_{k,f}$ is increasing, we see that $\sigma\leq r$. From (\ref{ineq: 3}) and the assumption that $\mathcal{Q}_{k,f}$ is nonincreasing, it follows
\begin{align*}
A_f(p,r)
&\leq \mathcal{Q}_{k,f}(r)v \\
&\leq \mathcal{Q}_{k,f}(\sigma)v \\
&=A_{k,f}(\sigma) \\
&=A_{k,f}\bigl(r_{k,f}(v)\bigr) \\
&=\mathcal{I}_{k,f}(v),
\end{align*}
which shows that
$$
A_f(p,r)
\leq
\mathcal{I}_{k,f}\bigl(V_f(p,r)\bigr).
$$

We now characterize the case of equality. Suppose first that $A_f(p,r)=\mathcal{I}_{k,f}(v)$. In particular, $A_f(p,r)=\mathcal{Q}_{k,f}(r)V_f(p,r)$. Equality in (\ref{ineq: 1}) shows that  $q_p(t)=q_p(r)$, since $q_p$ is nonincreasing. This implies that $q_p$ is constant on $(0,r)$. Moreover, since $\lim_{t\to0^+}q_p(t)=\alpha_{n-1}$, it follows that $q_p(t)=\alpha_{n-1}$. This shows that
$$
A(p,t)=\alpha_{n-1}s_k(t)^{n-1}=A_k(t).
$$
By Lemma \ref{lem:area-comparison}, $B_p(r)$ is isometric to the geodesic ball of radius $r$ in $\mathbb{M}_k^n$. The converse follows immediately.

\end{proof}

\noindent In particular, if $f(t)=c_k(t)=s_k'(t)$ in Theorem \ref{thm:weighted-isoperimetric}, then 
$$
\mathcal{Q}_{k,c_k}(r) = \frac{A_{k,c_k}(r)}{V_{k,c_k}(r)}=n\frac{c_k(r)}{s_k(r)},
$$
which shows 
$$\mathcal{Q}_{k,c_k}'(r) = -\frac{n}{s_k(r)^2}<0.$$
Combining this observation with Theorem~\ref{thm:weighted-isoperimetric}, we immediately obtain the following result.

\begin{corollary} 
\label{cor:ck-isoperimetric}
Let $(M^n,g)$ be an $n(\geq 2)$-dimensional complete Riemannian manifold satisfying $\Ric_M\geq(n-1)kg$ for $k \in \mathbb{R}$. Fix \(p\in M\). Denote $v:=V_{c_k}(p,r)$ for $r \in (0, \frac{R_k}{2})$. Then we have
\[
A_{c_k}(p,r)
\leq
n\omega_n^{1/n}
v^{\frac{n-1}{n}}
\left[
1-
k\left(
\frac{v}{\omega_n}
\right)^{\frac{2}{n}}
\right]^{1/2}.
\]
Moreover, equality holds if and only if $B_p(r)$ is isometric to the geodesic ball of radius $r$ in $\mathbb{M}_k^n$.
\end{corollary}

\begin{proof}
Let $ v:=V_{c_k}(p,r)$ and let $\sigma:=r_{k,c_k}(v)$ such that $V_{k,c_k}(\sigma)=v$. Then 
\begin{align*}
V_{k,c_k}(\sigma)
&=
\alpha_{n-1}
\int_0^\sigma c_k(t)s_k(t)^{n-1}\,dt \\
&=
\frac{\alpha_{n-1}}{n}s_k(\sigma)^n \\
&=
\omega_n s_k(\sigma)^n,
\end{align*}
where we used the fact that $\alpha_{n-1}=n\omega_n$. Hence
$$
s_k(\sigma)
=
\left(\frac{v}{\omega_n}\right)^{1/n}.
$$
Since $c_k(\sigma)^2+k s_k(\sigma)^2=1$, we obtain
$$
c_k(\sigma)^2
=
1-k\left(\frac{v}{\omega_n}\right)^{2/n}.
$$
Since $0\leq r<\frac{R_k}{2}$, we have 
$$
c_k(\sigma)
=
\left\{
1-k\left(\frac{v}{\omega_n}\right)^{2/n}
\right\}^{1/2}.
$$
It follows that
\begin{align*}
\mathcal{I}_{k,c_k}(v)
&=
A_{k,c_k}(\sigma) \\
&=
\alpha_{n-1}c_k(\sigma)s_k(\sigma)^{n-1} \\
&=
n\omega_n
\left\{
1-k\left(\frac{v}{\omega_n}\right)^{2/n}
\right\}^{1/2}
\left(\frac{v}{\omega_n}\right)^{(n-1)/n} \\
&=
n\omega_n^{1/n}v^{(n-1)/n}
\left\{
1-k\left(\frac{v}{\omega_n}\right)^{2/n}
\right\}^{1/2}.
\end{align*}
Applying Theorem~\ref{thm:weighted-isoperimetric} with $f=c_k$, we see that $A_{c_k}(p,r)\leq \mathcal{I}_{k,c_k}(v)$. Therefore we obtain
$$
A_{c_k}(p,r)
\leq
n\omega_n^{1/n}v^{(n-1)/n}
\left\{
1-k\left(\frac{v}{\omega_n}\right)^{2/n}
\right\}^{1/2}.
$$
Moreover, equality holds if and only if $B_p(r)$ is isometric to the geodesic ball of radius $r$ in $\mathbb{M}_k^n$. 
\end{proof}

\noindent We remark that, when $k=0$, since $s_0(r)=r$ and $c_0(r)=1$, Corollary~\ref{cor:ck-isoperimetric} reduces to
\[
A(p,r)
\leq
n\omega_n^{1/n}
V(p,r)^{\frac{n-1}{n}}.
\]
This is the sharp reverse isoperimetric inequality for geodesic balls in a manifold with nonnegative Ricci curvature.

\begin{corollary} \label{cor:power-weight}
Let $(M^n,g)$ be an $n(\geq 2)$-dimensional complete Riemannian manifold satisfying $\Ric_M\geq0$. Define \(\rho(x)=d(p,x)\) for some point $p\in M$. For \(\beta\geq0\), define
\begin{align*}
V_\beta(p,r) &:= \int_{B_p(r)}\rho(x)^\beta\dd\mu(x), \\ 
A_\beta(p,r) &:= \int_{S_p(r)}\rho(x)^\beta\dd A(x). 
\end{align*}
Then
\[
A_\beta(p,r)
\leq
(n+\beta)^{\frac{n+\beta-1}{n+\beta}}
\alpha_{n-1}^{\frac{1}{n+\beta}}
V_\beta(p,r)^{\frac{n+\beta-1}{n+\beta}}.
\]
Moreover, equality holds if and only if $B_p(r)$ is isometric to the Euclidean ball of radius $r$.
\end{corollary}

\begin{proof}
Note that $f(t)=t^\beta$ and $s_0(t)=t$. Thus
$$
\begin{aligned}
V_{0,\beta}(r)
&=
\alpha_{n-1}
\int_0^r t^\beta s_0(t)^{n-1}\,dt \\
&=
\alpha_{n-1}
\int_0^r t^{n+\beta-1}\,dt \\
&=
\frac{\alpha_{n-1}}{n+\beta}r^{n+\beta}.
\end{aligned}
$$
Moreover,
$$
A_{0,\beta}(r)
=
\alpha_{n-1}r^\beta s_0(r)^{n-1}
=
\alpha_{n-1}r^{n+\beta-1}.
$$
Thus we see that
$$
\mathcal{Q}_{0,\beta}(r)
:=
\frac{A_{0,\beta}(r)}{V_{0,\beta}(r)}
=
\frac{n+\beta}{r},
$$
which is strictly decreasing on $(0,\infty)$. Let $\sigma>0$ be the unique number such that $V_{0,\beta}(\sigma)=v$. Then
$$
\frac{\alpha_{n-1}}{n+\beta}\sigma^{n+\beta}=v,
$$
and hence
$$
\sigma
=
\left(
\frac{(n+\beta)v}{\alpha_{n-1}}
\right)^{\frac{1}{n+\beta}}.
$$
Therefore we obtain
$$
\begin{aligned}
\mathcal{I}_{0,\beta}(v)
&=
A_{0,\beta}(\sigma) \\
&=
\alpha_{n-1}\sigma^{n+\beta-1} \\
&=
\alpha_{n-1}
\left(
\frac{(n+\beta)v}{\alpha_{n-1}}
\right)^{\frac{n+\beta-1}{n+\beta}} \\
&=
(n+\beta)^{\frac{n+\beta-1}{n+\beta}}
\alpha_{n-1}^{\frac{1}{n+\beta}}
v^{\frac{n+\beta-1}{n+\beta}}.
\end{aligned}
$$
Theorem \ref{thm:weighted-isoperimetric} gives
$$
A_\beta(p,r)
\leq
(n+\beta)^{\frac{n+\beta-1}{n+\beta}}
\alpha_{n-1}^{\frac{1}{n+\beta}}
V_\beta(p,r)^{\frac{n+\beta-1}{n+\beta}}.
$$
Moreover, equality in Theorem~\ref{thm:weighted-isoperimetric} holds for a fixed $r>0$ if and only if $B_p(r)$ is isometric to the geodesic ball of radius $r$ in the Euclidean space $\mathbb{R}^n$.
\end{proof}

\vskip 0.3cm
\noindent
{\bf Acknowledgment: } The first author was supported by the National Research Foundation of Korea (NRF) grant RS-2024-00346651. The second author was supported by the National Research Foundation of Korea (NRF-2021R1A2C1003365).

\newpage

\noindent Jiewon Park\\
Department of Mathematical Sciences\\
Korea Advanced Institute of Science and Technology (KAIST)\\
Daejeon, South Korea\\
{\tt e-mail:jiewonpark@kaist.ac.kr}\\
URL: https://sites.google.com/view/jiewonpark/

\bigskip
\noindent Keomkyo Seo\\
Department of Mathematics and Research Institute of Natural Sciences\\
Sookmyung Women's University\\
Cheongpa-ro 47-gil 100, Yongsan-ku, Seoul, 04310, Korea \\
{\tt E-mail:kseo@sookmyung.ac.kr}\\
URL: http://sites.google.com/site/keomkyo/

\end{document}